\documentclass[12pt]{article} \textwidth 160mm\textheight 235mm \oddsidemargin-2mm \evensidemargin-2mm \topmargin-10mm \usepackage{amsfonts} \usepackage{amssymb} \usepackage{graphicx}
\usepackage{amsmath} \usepackage{amsthm} \usepackage{tikz} \usepackage{enumitem} \usepackage{cancel} \usepackage{todonotes}  \usepackage{cite} \usepackage{hyperref}
\hypersetup{colorlinks=true, urlcolor=blue} \usetikzlibrary{matrix} \usetikzlibrary{arrows,shapes} \usepackage{algorithm} \expandafter\let\expandafter\oldproof\csname\string\proof\endcsname
\let\oldendproof\endproof
\renewenvironment{proof}[1][\proofname]{%
\oldproof[\ttfamily \scshape \bf #1. ]%
}{\oldendproof}
 \def\ve{\varepsilon}    \def\emp{\emptyset}    \def \dist {{\rm dist}} \def\dom{{\rm dom}\,} \def\d{{\rm d}}     
       
    \def\O{{\cal O}}    \def\d{{\rm d}}  \def\Z{\mathbb{Z}}
       \def\ox{\overline{x}} 
\def\oq{\overline{q}} \def\oy{\overline{y}}      
\def\disp{\displaystyle}      
 
 \def\tto{\;{\lower 1pt \hbox{$\rightarrow$}}\kern -10pt
\hbox{\raise 2pt \hbox{$\rightarrow$}}\;}  \def\Vt{\vartheta} \def\Tilde{\widetilde}  \def\ra{\rangle} \def\la{\langle} \def\ve{\varepsilon} 
 \def\R{\mathbb{R}}  \def\X{\mathbb{X}} \def\Y{\mathbb{Y}} \def\ox{\bar{x}} \def\oy{\bar{y}}   
 \def\ou{\bar{u}}   \def\oq{\bar{q}}    
 \def\gph{\mbox{\rm gph}\,}   \def\dom{\mbox{\rm dom}\,}  
      
   \def\dn{\downarrow} \def\O{\Omega}
\def\ph{\varphi} \def\emp{\emptyset}   \def\lm{\lambda} \def\gg{\gamma} \def\dd{\delta}  \def\Th{\Theta}  
\def\vt{\vartheta}   \def\ss{\scriptsize }   
  \def\sce{\setcounter{equation}{0}}   
\begin{document}
\newtheorem{Theorem}{Theorem}[section]
\newtheorem{Proposition}[Theorem]{Proposition}
\newtheorem{Remark}[Theorem]{Remark}
\newtheorem{Lemma}[Theorem]{Lemma}
\newtheorem{Corollary}[Theorem]{Corollary}
\newtheorem{Definition}[Theorem]{Definition}
\newtheorem{Example}[Theorem]{Example}
\renewcommand{\theequation}{{\thesection}.\arabic{equation}}
\renewcommand{\thefootnote}{\fnsymbol{footnote}}

\begin{center} {\bf \Large Optimality Conditions for Variational Problems\\in Incomplete Functional Spaces}\\[2ex] ASHKAN MOHAMMADI\footnote{Department of Mathematics and Statistics, Georgetown University, Washington, DC 20007, USA (ashkan.mohammadi@georgetown.edu).} BORIS S. MORDUKHOVICH\footnote{Department of Mathematics, Wayne State University, Detroit, MI 48202, USA
(boris@math.wayne.edu). Research of this author was partly supported by the National Science Foundation under grants DMS-1512846 and DMS-1808978, by the USA Air Force Office of Scientific
Research under grant \#15RT04, and by the Australian Research Council under Discovery Project DP-190100555.}\\[2ex] {\bf Dedicated to Professor Franco Giannessi in the occasion of his 85th
birthday} \end{center} \vspace*{0.05in} \small{\bf Abstract.} This paper develops a novel approach to necessary optimality conditions for constrained variational problems defined in generally
incomplete subspaces of absolutely continuous functions. Our approach consists of reducing a variational problem to a (nondynamic) problem of constrained optimization in a normed space and then
applying the results recently obtained for the latter class by using generalized differentiation. In this way we derive necessary optimality conditions for nonconvex problems of the calculus of
variations with velocity constraints under the weakest metric subregularity-type constraint qualification. The developed approach leads us to a short and simple proof of first-order necessary
optimality conditions for such and related problems in broad spaces of functions including those of class ${\cal C}^k$ as $k\ge 1$.\\[1ex] {\bf Key Words.} Calculus of variations, constrained
optimization, optimal control, necessary optimality conditions, variational analysis, generalized differentiation\\[1ex] {\bf Mathematics Subject Classification (2000)} 49K24, 49J52, 49J53, 90C48

\normalsize
\section{Introduction}\label{intro}\sce

The classical calculus of variations primarily deals with minimizing integral functionals on classes of smooth curves that mainly belong to the spaces of continuously differentiable or twice
continuously differentiable functions; see the fundamental monographs by Bolza \cite{bolza}, Tonelli \cite{tonelli}, and Bliss \cite{bliss} as well as extensive further developments on the
subject. Although the aforementioned and other spaces used in the calculus of variations are {\em incomplete}, this does not create any obstacles in deriving necessary conditions for optimal
solutions to such problems due to the employed (Lagrangian) {\em method of variation}.

Modern variational analysis offers powerful techniques to derive necessary optimality conditions in problems of dynamic optimization including those in the calculus of variations and optimal
control. This machinery is based on advanced {\em variational principles} and {\em approximation/limiting procedures}, which are applied to general problems governed, in particular, by
differential inclusions where the method of variations and its modifications are not applicable. For various techniques in this vein we refer the reader to the books \cite{c,i,m06,s,v} with the
bibliographies and commentaries therein. However, there is a price to pay: all such methods unavoidably require the {\em completeness} of the space in question and thus cannot be applied to
optimization problems in spaces of smooth functions that have been traditionally considered in the calculus of variations.

This paper is devoted to developing a novel approach of variational analysis and generalized differentiation to derive necessary optimality conditions in constrained problems of dynamic
optimization over curves belonging to a {\em prescribed normed space} located between the collections of absolutely continuous and infinitely differentiable functions. The suggested approach is
based on the reduction of a given dynamic optimization problem to an infinite-dimensional nondynamic problem of {\em constrained optimization} for which necessary optimality conditions have been
recently obtained in our paper \cite{mm} in arbitrary normed spaces under weak constraint qualifications.

Although this approach works in more general frameworks of constrained {\em dynamic optimization}, for simplicity we concentrate here on an {\em extended Bolza problem} of the calculus of
variations considered in the aforementioned (generally incomplete) spaces of curves subject to endpoint and {\em hard/pointwise constraints on velocity} functions that depend on the current state
position. Pointwise velocity constrains have been recognized as the most challenging ones in the calculus of variations. Even in a modern setting with the usage of an advanced variational
technique largely different from the method of variations, the necessary optimality conditions for strong local minimizers in problems of the calculus of variations with pure velocity constraints
in the complete space of absolutely continuous functions are obtained under the restrictive ``Interiority Hypothesis" in the most recent Clarke's book \cite[Theorem~18.1]{c}.

The {\em reduction method} developed in this paper allows us to represent the original variational problem in an equivalent form of nondynamic infinite-dimensional constrained optimization and
then apply the necessary optimality conditions to the latter problem established in \cite{mm}. In this way we present a rather simple derivation of necessary optimality conditions for strong
local (in a generalized sense) minimizers of the extended Bolza problem under consideration defined in generally incomplete subspaces of absolutely continuous functions. The obtained necessary
optimality conditions consist of the {\em Euler-Lagrange equation}, the {\em Weierstrass-Pontryagin maximization condition}, and the {\em transversality inclusion} in the {\em
qualified/normal/KKT form} established under the weakest constraint qualification of the {\em metric subregularity} type.\vspace*{0.05in}

The rest of the paper is organized as follows. Section~\ref{prel} contains the required definitions and preliminaries from variational analysis and generalized differentiation used in the
formulations and proofs of the subsequent results. We present here the underlying theorem from \cite{mm} giving us necessary optimality conditions for infinite-dimensional constrained
optimization problems to which we reduce the extended Bolza problem of our study.

Section~\ref{main} starts with the formulation and discussion of this extended version of the Bolza problem with endpoint and pointwise velocity constraints. Then we formulate the aforementioned
necessary optimality conditions for the extended Bolza problem that are proved in the remaining part of the paper by the reduction to constrained optimization.

All the reduction steps are furnished in Section~\ref{conic}, which is the most technical part of the paper while containing results of their own interest. Using this reduction and the obtained
optimality conditions in nondynamic constrained optimization, we complete in Section~\ref{sec:proof} the derivation of the necessary optimality conditions for the extended Bolza problem that are
formulated in Section~\ref{main}. Section~\ref{concl} summarizes the main achievements of the paper and discusses some topics of our future research.

\section{Basic Definitions and Preliminaries}\label{prel}\sce

First we recall some standard notation of variational analysis used in the paper. Unless otherwise stated, $\X$ and $\Y$ stand for {\em normed spaces} with the generic symbol $\|\cdot\|$ for
norms and $\la\cdot,\cdot\ra$ for scalar products between the spaces in question and their topological duals.

Given an extended-real-valued function $\ph\colon\X\to(-\infty,\infty]$ with the domain $\dom\ph:=\{x\in\X\;|\;\ph(x)<\infty\}$, the (Dini-Hadamard) {\em subderivative} of $\ph$ at
$\ox\in\dom\ph$ is the function ${\mathrm d}\ph(\ox)\colon\X\to[-\infty,\infty]$ defined by \begin{equation}\label{subder} {\mathrm d}\ph(\ox)(\ou):=\liminf_{\substack{t\dn
0\\u\to\ou}}{\frac{\ph(\ox+tu)-\ph(\ox)}{t}},\quad\ou\in\X,\end{equation} where the limit $u\to\ou$ in \eqref{subder} can be equivalently omitted if $\ph$ is locally Lipschitzian around $\ox$.
The latter form reduces to the {\em directional derivative} of $\ph$ in the direction $\ou$ provided that the full limit in \eqref{subder} exists. It is well known that the {\em G\^ateaux
differentiability} of $\ph$ at an interior point $\ox$ of the domain corresponds to the existence of the directional derivative in any direction and its linearity with respect to the direction
variable.

Turning next to sets, we associate with any nonempty subset $\O\subset\X$ the {\em indicator function} $\dd_\O(x)$ of $\O$ that equals $0$ if $x\in\O$ and $\infty$ otherwise, and the {\em
distance function} $\dist(x;\O)$ of $\O$ defined as usual by \begin{equation*}{\rm dist}(x;\O):=\inf\big\{\|x-u\|\;\big|\;u\in\O\big\}.\end{equation*}  The latter function is Lipschitz continuous
on $\X$ with Lipschitz constant $\ell=1$. Since the main goal of this paper is to illuminate the suggested approach to deriving necessary optimality conditions in problems of dynamic optimization
by reducing them to nondynamic constrained optimization without much of technical complications, we are not going to involve here tangent and normal cone constructions for nonconvex sets. The
only {\em normal cone} used in what follows is the classical one for convex sets $\O\subset\X$ defined by \begin{equation}\label{nc} N_\O(\ox):=\big\{x^*\in\X^*\;\big|\;\la x^*,x-\ox\ra\le
0\;\mbox{ for all }\;x\in\O\big\}\end{equation} if $\ox\in\O$ with $N_\O(\ox):=\emp$ otherwise. The set of normals in \eqref{nc} is obviously convex and closed in the weak$^*$ topology of the
dual space $\X^*$.

Considering further a set-valued mapping $F\colon\X\tto\Y$ with the graph $\gph F:=\{(x,y)\in\X\times\Y\;|\;y\in F(x)\}$, recall that $F$ is {\em metrically regular} around $(\ox,\oy)\in\gph F$
if there exist a constant $\kappa>0$ and neighborhoods $U$ of $\ox$ and $V$ of $\oy$ such that we have \begin{equation}\label{metreq}\dist \big(x ; F^{-1}(y)\big)\le\kappa\,\dist
\big(y;F(x)\big)\;\mbox{ for all }\;(x,y)\in U\times V.\end{equation} If $y=\oy$ in \eqref{metreq}, the mapping $F$ is called to be {\em metrically subregular} at $(\ox,\oy)$. The reader is
referred to the books \cite{i,m06,rw} for more information about these and equivalent properties of set-valued mappings with their broad applications in variational analysis.\vspace*{0.05in}

Now we formulate a class of (nondynamic) constrained optimization problems, which was studied in our previous paper \cite{mm} with deriving various types of primal and dual necessary optimality
conditions for their local minimizers. Given a cost function $J\colon\X\to\R$, a constraint mapping $f\colon\X\to\Y$ between arbitrary normed spaces, and a constraint set $\Th\subset\Y$, the
basic {\em constrained optimization problem} is defined as follows: \begin{equation}\label{op} \mbox{minimize }\;J(x)\;\mbox{ subject to }\;f(x)\in\Theta\end{equation} with the set of {\em
feasible solutions} denoted by $\O:=\{x\in\X\;|\;f(x)\in\Theta\}$. Among the necessary optimality conditions obtained for \eqref{op} in \cite{mm}, we select the dual one established in the
refined KKT form under the following constraint qualification.

\begin{Definition}[\textbf{metric subregularity constraint qualification}]\label{defmscq} Let $\ox$ be a feasible solution \eqref{op}. Then we say that the {\sc metric subregularity constraint
qualification} $($MSCQ$)$ holds at $\bar x$ if the set-valued mapping $x\mapsto f(x)-\Theta$ is metrically subregular at $(\bar x,0)$, i.e., there exists a constant $\kappa>0$ and a neighborhood
$U$ of $\ox$ such that \begin{equation}\label{mscq} \dist(x;\O)\le\kappa\;\dist\big(f(x);\Theta\big)\;\mbox{ for all }\;x\in U. \end{equation}\end{Definition}

Note that the replacement of the metric subregularity of the mapping $x\mapsto f(x)-\Theta$ at $(\ox,0)$ in Definition~\ref{defmscq} by the {\em metric regularity} of this mapping around the pair
$(\ox,0)$ brings us to a significantly more restrictive constraint qualification, which reduces to the well-known ones for particular classes of optimization problems (e.g., the {\em
Mangasarian-Fromovitz constraint qualification} in nonlinear programming, the {\em Robinson constraint qualification} in conic programming, etc.). This follows from applying the {\em Mordukhovich
coderivative criterion} to the mapping $f-\Th$ around $(\ox,0)$ for sets $\Th$ that appear in particular constraint systems; see \cite{m93,m06,rw}. Regarding the more subtle MSCQ, its
relationships with other constraint qualifications, and various applications, we refer the reader to, e.g., \cite{g,gm,go,mm,mms} with the additional details and discussions.\vspace*{0.05in}

Finally in this section, we present the necessary optimality conditions for the constrained problem \eqref{op} in normed spaces used in what follows. Note that the following theorem is a special
case of \cite[Theorem~7.3]{mm}, where the optimality conditions are established under more general assumptions. However, we confine ourselves to the ones below to simplify the subsequent
derivation of necessary conditions for the extended Bolza problem formulated in the next section. Recall that $A^*$ indicates the adjoint operator of the linear operator $A$, which reduces to the
matrix transposition in finite dimensions.

\begin{Theorem}[\textbf{necessary conditions for constrained optimization}]\label{dualnopc} Let $\X$ and $\Y$ be arbitrary normed spaces, and let $\ox\in\O$ be a local $($in the norm of the space
$\X)$ minimizer of problem \eqref{op}, where $\Theta\subset\Y$ is convex and locally closed around $f(\ox)$, where $J\colon\X\to\R$ is G\^ateaux differentiable at $\ox$ and locally Lipschitzian
around this point with Lipschitz constant $\ell>0$, and where $f\colon\X\to\Y$ is continuously Fr\'echet differentiable around $\ox$. Assume in addition that MSCQ \eqref{mscq} holds at $\ox$ with
some constant $\kappa>0$. Then we have the following necessary optimality conditions: \begin{equation}\label{necessaryop}\mbox{there exists }\;\lambda\in\Y^*\;\mbox{ such that
}\;\left\{\begin{matrix} 0 = \nabla J(\ox) + \nabla f(\ox)^*\lambda, \\\\ \lambda \in N_{\Theta}\big(f(\ox)\big),\;\| \lambda\|\leq \kappa \ell, \\ \end{matrix}\right. \end{equation} where the
same symbol $\nabla$ is used for both G\^ateaux and Fr\'echet derivatives.\end{Theorem}

Note that if $J$ is also continuously Fr\'echet differentiable around $\ox$, then we can set $\ell:=\|\nabla J(\ox)\|$; see \cite[Corollary~7.5]{mm}.


\section{Extended Bolza Problem with Velocity Constraints}\label{main}\sce

In this section we define a constrained variational problem written in an extended form of the Bolza problem of the calculus of variations, while in the presence of pointwise velocity constraints
depending on the current curve position. Feasible curves in this problem belong to a prescribed generally incomplete subspace of functions situated between the spaces of infinite differentiable
and absolutely continuous ones.

To formulate the problem of our study, consider the terminal cost $\ph\colon\R^n\times\R^n\to\R$, the running cost $\vt\colon[0,T]\times\R^n\times\R^n\to\R$, the constraint mappings
$g\colon[0,T]\times\R^n\to\R^n$, the dynamic constraint set $\Omega_1\subset\R^n$, and the endpoint constraint set $\O_2\subset\R^n\times\R^n$, where the time $T>0$ is fixed. Let $\X$ be an
arbitrary normed space such that \begin{equation}\label{space} {\cal C}^\infty([0,T];\R^n)\subset\X\subset AC([0,T];\R^n),\end{equation} where ${\cal C}^\infty([0,T];\R^n)$ stands for the
standard space of functions $x\colon[0,T]\to\R^n$ infinite differentiable on $[0,T]$, and where $AC([0,T];\R^n)$ indicates the space of all functions $x(\cdot)$ that are absolutely continuous on
$[0,T]$ with the norm \begin{equation}\label{acnorm}\| x(\cdot)\|_{ac}:= \| x(0) \| + \int_{0}^{T}  \| \dot{x} (t) \|\; dt. \end{equation} Both inclusions in \eqref{space} can be nonstrict, i.e.,
the extreme cases of $\X={\cal C}([0,T];\R^n)$ and $\X=AC([0,T];\R^n)$ are also acceptable. Unless otherwise stated, the norm on the space $\X$ is given by \eqref{acnorm}. In particular, the
choice of $\X$ in \eqref{space} includes the incomplete spaces ${\cal C}^k([0,T];\R^n)$ of $k$-times differentiable vector functions with $k=1,2,\ldots$, which are typically encountered in the
classical calculus of variations. Our basic {\em extended Bolza problem} is formulated as follows: \begin{eqnarray}\label{bolzap} \disp &\mbox{minimize}&
J\big(x(\cdot)\big):=\quad\ph\big(x(0),x(T)\big)+\int_{0}^{T}\Vt\big(t,x(t),\dot{x}(t)\big)\,dt\;\mbox{ over }\;x(\cdot)\in\X\nonumber \\&\mbox{subject to}&\dot{x}(t) + g\big(t,x(t)\big)\in\O_1
\;\mbox{ a.e. } \; t \in [0,T],\quad\big(x(0),x(T)\big)\in\O_2. \disp \end{eqnarray}

We say as usual that $x(\cdot)\in\X$ is a {\em feasible solution} to
problem \eqref{bolzap} if $x(\cdot)$ satisfies all the constraints
in this problem and gives a finite value of the cost functional
therein. The set of feasible solutions to \eqref{bolzap} is denoted
by ${\cal S}$. A feasible solution $\ox=\ox(\cdot)$ is said to be an
{\em $\X$-strong local minimizer} of \eqref{bolzap} if there exists
$\ve>0$ such that \begin{equation}\label{st-min}
J(\ox(\cdot)\big)\le J\big(x(\cdot)\big)\;\mbox{ for any
}\;x(\cdot)\in{\cal S}\;\mbox{ with
}\;\|x(\cdot)-\ox(\cdot)\|_{ac}<\ve. \end{equation} Note that the
notion of $\X$-strong local minimizers defined in \eqref{st-min} is
different (even for $\X=AC([0,T];\R^n)$) from the standard notion of
strong minimizers in the calculus of variations, where the
$\X$-closeness in \eqref{st-min} is replaced by the closeness in the
uniform topology of the space ${\cal C}([0,T];\R^n)$; cf.\
\cite{bliss,bolza,c}. In fact, $\X$-strong local minimizers occupy a
(proper) {\em intermediate} position between weak and strong
minimizers of the calculus of variations; see
\cite[Section~6.1]{m06} for more discussions, examples, and
references.\vspace*{0.05in}

As mentioned above, our intention is to reduce the extended Bolza problem of dynamic optimization \eqref{bolzap} in the normed space $\X$ to the nondynamic one \eqref{op} in a suitable
functional space in order to apply to the latter the necessary optimality conditions established in Theorem~\ref{dualnopc}. To proceed in this way, we have to formulate appropriate assumptions on
the given data of \eqref{bolzap} that ensure the fulfillment of the required assumptions of Theorem~\ref{dualnopc} for the reduced constrained optimization problem \eqref{op}. Let us impose the
following assumptions on the initial data $\ph$, $\vt$, $g$, $\Omega_1$, and $\Omega_2$ of \eqref{bolzap} around the reference optimal solution ($\X$-strong local minimizer) $\ox=\ox(\cdot)\in\X$
of this problem. \\[1ex] {\bf(H1)} The terminal cost $\ph(x_0,x_T)$ is continuously differentiable around $(\ox(0),\ox(T))$.\\[1ex] {\bf(H2)} The running cost $\vt(t,x,v)$ is measurable in $t$ on
$[0,T]$, continuously differentiable in $(x,v)$ around $(\ox(t),\dot\ox(t))$ for a.e.\ $t\in[0,T]$, and locally Lipschitzian with respect to $(x,v)$ around $(\ox(\cdot),\dot\ox(\cdot))$ in the
$\X$-norm \eqref{acnorm} with a summable Lipschitz modulus $\ell(t)$ on $[0,T]$. This means that there exists $\ve>0$ such that for all $x_1(\cdot),x_2(\cdot)\in\X$ near $\ox(t)$ we have
\begin{equation*} \big\|\vt\big(t,x_1(t),\dot x_1(t)\big)-\vt\big(t,x_2(t),\dot x_2(t)\big)\big\|\le\ell(t)\big(\|x_1(t)-x_2(t)\|+\|\dot x_1(t)-\dot x_2(t)\|\big)\;\mbox{ a.e. }\;t\in[0,T]
\end{equation*} provided that $\|x_1(\cdot)-x_2(\cdot)\|_{ac}\le\ve$. For simplicity we suppose that $\ell(t)\equiv\ell$ on $[0,T]$.\\[1ex]  {\bf(H3)} The constraint sets $\O_1$ and $\O_2$ are
convex and locally closed around $\ox(\cdot)$ in the spaces $\R^n$ and $\R^n\times\R^n$, respectively. This means that there exist closed balls around $\dot{\ox}(t)+g(t,\ox(t))$ for a.e.\
$t\in[0,T]$ and $(\ox(0),\ox(T))$ in the corresponding finite-dimensional spaces such that the intersections of these balls with $\O_1$ and $\O_2$ are closed.\\[1ex] {\bf(H4)} The constraint
mapping $g(t,x)$ is continuously differentiable in $x$ and measurable in $t$ together with its derivative $\nabla_x g(t,x)$. Furthermore, both $g(t,x)$ and $\nabla_x g(t,x)$ are essentially
bounded on $[0,T]$ for all $x(\cdot)$ around $\ox(\cdot)$, where the localization is understood similarly to the description in (H3).\\[1ex] {\bf(H5)} There exists a constant $\kappa>0$ such that
for all $x(\cdot)\in\X$ sufficiently close to $\ox(\cdot)$ we have the constraint qualification \begin{equation}\label{bolzamcq} \dist\big(x(\cdot);{\cal S}\big)\le\kappa\int_{0}^{T}
\dist\big(\dot{x}(t)+g(t,x(t));\O_1\big)\,dt+\kappa\,\dist\big((x(0),x(T));\O_2\big).\end{equation}

As observed by one of the referees, in the case where
$\O_2=\O_0\times\R^n$, assumption (H5) follows from the
Filippov-Gronwall inequality; see, e.g., \cite[Proposition~1 of
Chapter 2]{aubin}. Furthermore, we'll see below that the local
Lipschitz continuity assumption on the running cost in the $\X$-norm
imposed in (H2), which is generally different from the standard
local Lipschitz continuity of $\vt(t,x,v)$ in
$(x,v)\in\R^n\times\R^n$, allows us to deal with $\X$-strong local
minimizers $\ox(\cdot)\in\X$ of the extended Bolza problem
\eqref{bolzap}.\vspace*{0.05in}

Here is the formulation of necessary optimality conditions for strong local minimizers of \eqref{bolzap}, which are proved in the next two sections.

\begin{Theorem}[\textbf{necessary optimality conditions for the extended Bolza problem}]\label{bolza} Let $\ox(\cdot)\in\X$ be an $\X$-strong local minimizer \eqref{st-min} of the extended Bolza
problem \eqref{bolzap} under the fulfillment of the assumptions {\rm(H1)}--{\rm(H5)} around $\ox(\cdot)$. Then there exists an adjoint arc $p(\cdot)\in AC([0,T];\R^n)$ for which the following
conditions are satisfied:\\[1ex] The {\sc Euler-Lagrange equation} for a.e. $t\in[0,T]$: \begin{equation}\label{euler-lag}\dot{p}(t)-\nabla_{x}
g(t,\ox(t))^*p(t)=\nabla_x\Vt\big(t,\ox(t),\dot{\ox}(t)\big)- \nabla_{x}g\big(t,\ox(t)\big)^*\nabla_v\Vt\big(t,\ox(t),\dot{\ox}(t)\big).\end{equation} The {\sc Weierstrass-Pontryagin maximization
condition} for a.e. $t\in[0,T]$:\begin{equation}\label{maxp}\big\la p(t)-\nabla_v\Vt\big(t,\ox(t),\dot{\ox}(t)\big),\dot{\ox}(t)+g\big(t,\ox(t)\big)\big\ra=\max_{w\in\O_1}\big\la
p(t)-\nabla_v\Vt\big(t,\ox(t), \dot{\ox}(t)\big),w\big\ra. \end{equation} The {\sc transversality inclusion}:
\begin{equation}\label{transv}\big(p(0),-p(T)\big)\in\nabla\ph\big(\ox(0),\ox(T)\big)+N_{\O_2}\big(\ox(0),\ox(T)\big). \end{equation} \end{Theorem}

Observe that if $g\equiv0$ in \eqref{bolzap}, i.e., we have the problem of Bolza with pure velocity constraints and if $\X=AC([0,T];\R^n)$, then the maximization condition \eqref{maxp} reduces to
the Weierstrass condition obtained \cite[Theorem~18.1]{c} under the ``Interiority Hypothesis" that is much more restrictive than the qualification condition \eqref{bolzamcq}. On the other hand,
condition \eqref{maxp} corresponds to the extensions of the Pontryagin maximum principle \cite{pont} to variational problems governed by differential inclusions $\dot x\in F(t,x)$ with
$F(t,x):=\O_1-g(t,x)$; see, e.g., \cite{c,i,m06,s,v} for various results and proofs in complete spaces of functions. Our simple reduction proof of Theorem~\ref{bolza} is given in the next two
sections.


\section{Reduction to Constrained Optimization}\label{conic}\sce

First we rewrite the extended Bolza problem \eqref{bolzap} in the form of the constrained optimization problem \eqref{op} in the normed space $\X$ of functions $x=x(\cdot)$ taken from
\eqref{space} and endowed with the norm \eqref{acnorm}. The data of this problem are defined in terms of \eqref{bolzap} by \begin{eqnarray}\label{dataoc}\disp
J(x)&:=&\ph\big(x(0),x(T)\big)+\int_{0}^{T}\Vt\big(t , x(t), \dot{x} (t)\big)\,dt,\nonumber \\f(x)&:=& \big(\dot{x}+ g(\cdot,x(\cdot)),\;(x(0),x(T))\big),\\\nonumber\Theta &:=&\big\{y(\cdot)\in
L^1([0,T];\R^n)\;\big|\;y(t) \in\O_1\;\mbox{ a.e. }\;t\in[0,T] \big\}\times\O_2.\disp \end{eqnarray}

It is easy to see that the set $\Th$ in \eqref{dataoc} is convex and locally closed around $f(\ox(\cdot))$ in the space under consideration. To proceed with the applications of
Theorem~\ref{dualnopc}, we need to check that the mappings $J$ and $f$ from \eqref{dataoc} satisfy the Lipschitz continuity and differentiability assumptions imposed in the latter theorem.

Let us begin by observing that the space $AC([0,T];\R^n)$, which contains $\X$ and is equipped with norm \eqref{acnorm}, is isometric to the space $L^1([0,T];\R^n)\times\R^n$ via the isometry
$x\mapsto(\dot{x},x(0))$, and hence the dual space of $AC([0,T];\R^n)$ can be identified with $L^{\infty}([0,T];\R^n)\times\R^n$. This tells us that the space $\X$ is densely embedded into
$L^1([0,T];\R^n)\times\R^n $, which tells us that the dual space of $\X$ can be identified with $L^{\infty}([0,T];\R^n)\times\R^n$. Furthermore, using the integral representation $$ T x(0)=
\int_0^T x(t)\,dt - \int_0^T\bigg(\int_0^s \dot{x} (t)\,dt\bigg)\,ds $$ implies that the norm $\|\cdot\|_{ac}$ in \eqref{acnorm} is equivalent to \begin{equation}\label{w11norm} \| x \|_{1,1}: =
\int_{0}^{T} \| x (t) \| \; dt + \int_{0}^{T} \| \dot{x}(t) \|\,dt. \end{equation} In fact, we have the precise equivalence relationships \begin{equation}\label{normeq} \frac{1}{ 1+T}\;
\|x\|_{1,1} \leq \|x\|_{ac} \leq \frac{2T +1}{T} \; \| x \|_{1,1}\;\mbox{ for all }\;x \in AC([0,T];\R^n).  \end{equation} Moreover, it follows from the inequality $\disp\bigg\| \int_{0}^{T}  x (t)\,dt\bigg\|
\leq \int_{0}^{T} \big\| x (t)\big\|\,dt$ that \begin{equation}\label{maxnorm} \|x\|_\infty:=\sup\big\{|x_i(t)|\;\big|\;i=1,\ldots,n,\;0\le t\le T\big\}\le \frac{2 + 2T}{T} \; \|x\|_{1,1} \end{equation}

The next two theorems of their own interest verify the Lipschitz continuity and differentiability assumptions of Theorem~\ref{dualnopc} in the case where $J$ and $f$ are taken from \eqref{dataoc}
under the assumptions imposed in (H1)--(H4). Observe that both theorems do not require that $\ox(\cdot)$ is an $\X$-strong local minimizer of \eqref{bolzap} as formulated in latter assumptions
while being hold for broader classes of curves $\ox(\cdot)\in\X$ satisfying the corresponding properties.\vspace*{0.05in}

We start with verifying the required properties of the cost functional $J$ in \eqref{bolzap}.

\begin{Theorem}[\textbf{properties of the cost functional}]\label{fisdef} Let the assumptions ${\rm(H1)}$ and ${\rm(H2)}$ be satisfied around a given curve $\ox=\ox(\cdot)\in\X$ with
$J(\ox)<\infty$. Then the cost functional in \eqref{bolzap} is locally Lipschitzian around $\ox$ and G\^ateaux differentiable at this point of $\X$ with the following calculation of its G\^ateaux
derivative at $\ox$ in any direction $u=u(\cdot)\in\X$: \begin{equation}\label{gJ}\begin{array}{ll}\nabla J(\ox)(u)=&\big\la\nabla\ph\big(\ox(0),\ox(T)\big),\big(u(0),u(T)\big)\big\ra\\\\
&+\disp\int_{0}^{T}\Big[\big\la\nabla_x\Vt\big(t,\ox(t),\dot{\ox}(t)\big),u(t)\big\ra+\la\nabla_v\Vt\big(t,\ox(t),\dot{\ox}(t)\big),\dot{u}(t)\big\ra\Big]\,dt.\end{array} \end{equation}
\end{Theorem} \begin{proof} First we consider only the integral part \begin{equation}\label{I}I(x):=\int_{0}^{T} \Vt\big(t, x(t), \dot{x} (t)\big)\,dt\end{equation} of the cost functional in
\eqref{dataoc} and establish its local Lipschitz continuity around $\ox$ as well as the G\^ateaux differentiability at $\ox$ with the G\^ateaux derivative representation
\begin{equation}\label{gI} \nabla
I(\ox)(u)=\disp\int_{0}^{T}\Big[\big\la\nabla_x\Vt\big(t,\ox(t),\dot{\ox}(t)\big),u(t)\big\ra+\la\nabla_v\Vt\big(t,\ox(t),\dot{\ox}(t)\big),\dot{u}(t)\big\ra\Big]\,dt\end{equation} for all
$u(\cdot)\in\X$. To proceed, pick any $x(\cdot)\in\X$ near $\ox(\cdot)$ and deduce from (H2) that \begin{equation*}\big|\vt\big(t,x(t),\dot
x(t)\big)\big|\le\big|\vt\big(t,\ox(t),\dot\ox(t)\big)\big|+\ell\|x(t)-\ox(t)\|+\ell\|\dot x(t)-\dot\ox(t)\|. \end{equation*} Taking the integral over $[0,T]$ from both sides of the above
inequality and using the equivalent norm description \eqref{w11norm} implies that $I(x)$ is finite, i.e., the integral functional \eqref{I} is real-valued around $\ox(\cdot)$ in the $\X$-norm.

Next we verify that the integral functional \eqref{I} is Lipschitz continuous around $\ox(\cdot)$ in the $\X$-norm. Take any $x_1(\cdot),x_2(\cdot)\in\X$ from the $\ve$-neighborhood of
$\ox(\cdot)$ in the $\X$-norm where (H2) holds. Combining this assumption with \eqref{normeq} gives us the estimates \begin{eqnarray*} \disp| I(x_1) -  I(x_2) |& \leq & \int_{0}^{T} \big\|
\Vt\big(t , x_1(t), \dot{x}_1 (t)\big) - \Vt\big(t,x_2(t), \dot{x}_2 (t)\big)\big\|\,dt \\ &\leq & \int_{0}^{T} \big( \ell \| x_1(t) - x_2(t) \| + \ell \| \dot{x}_1 (t) - \dot{x}_2 (t) \|
\big)\,dt \\ &\leq & \ell \|x_1 - x_2 \|_{\ss 1,1} \leq \ell(1+T) \; \| x_1 - x_2\|_{ac},\disp\end{eqnarray*} which ensure the claimed local Lipschitz continuity of the functional $I$ in $\X$
around $\ox(\cdot)$.

To show now that $I$ is G\^ateaux differentiable at $\ox(\cdot)$, fix any direction $u(\cdot)\in\X$. Using definition \eqref{subder} of the Dini-Hadamard subderivative and the established local
Lipschitz continuity of $I$ at $\ox(\cdot)$ in the space $\X$ under consideration, we find a decreasing sequence of positive number $\tau_k \dn 0$ as $k\to\infty$ such that \begin{eqnarray*}
\disp \d I(\ox)(u) &=& \liminf_{\tau \dn 0} \frac{I(\ox + \tau u) - I(\ox)}{\tau}= \lim_{k\to \infty} \frac{I(\ox + \tau_k u) - I(\ox)}{\tau_k} \\ &=& \lim_{k \to \infty} \int_{0}^{T}
\bigg(\frac{\Vt\big(t , \ox(t)+ \tau_k u(t), \dot{\ox} (t)+ \tau_k \dot{u}(t)\big) - \Vt\big(t , \ox(t), \dot{\ox} (t)\big)}{\tau_k}\bigg)\,dt \\ &=& \int_{0}^{T}\bigg(\lim_{k \to \infty}
\frac{\Vt\big(t ,\ox(t)+ \tau_k u(t), \dot{\ox} (t)+ \tau_k \dot{u}(t)\big) - \Vt\big(t , \ox(t), \dot{\ox} (t)\big)}{\tau_k}\bigg)\,dt \\ &=& \int_{0}^{T}\big\la\nabla_x\Vt\big(t , \ox(t),
\dot{\ox} (t)\big) , u(t)\big\ra +\big\la\nabla_v\Vt\big(t , \ox(t), \dot{\ox}(t)\big) , \dot{u}(t)\big\ra\,dt, \disp \end{eqnarray*} where in the third line we interchange the limit and integral
signs by using the Lebesgue dominated convergence theorem with taking into account the integrand function in the second line is dominated by $\ell\big (\|u(t)\|+\|\dot{u}(t)\|\big)$ for a.e. $t
\in [0,T]$ due to Lipschitzian assumption in (H2). The last line above comes from the smoothness assumption on $\Vt$ imposed in (H2). By similar arguments we arrive at the upper limit
representation \begin{equation*} \limsup_{\tau \dn 0} \frac{I(\ox + \tau u) - I(\ox)}{\tau} = \int_{0}^{T}\big\la\nabla_x\Vt\big(t, \ox(t),\dot{\ox} (t)\big), u(t)\big\ra +
\big\la\nabla_v\Vt\big(t , \ox(t), \dot{\ox} (t)\big) , \dot{u}(t)\big\ra\,dt. \end{equation*} Unifying the latter with the previous one for the lower limit proves the existence of the classical
directional derivative  of $I$ at $\ox$ given by \begin{equation*} \d I(\ox)(u)= \lim_{\tau \dn 0} \frac{I(\ox + \tau u)-I(\ox)}{\tau}\;\mbox{ for any }\;u(\cdot)\in\X,\end{equation*} which is
clearly linear and continuous with respect to the direction variable with $| \d I(x) (u) | \leq \ell \|u \|$ for each $u\in\X$. This shows that the integral functional \eqref{I} is G\^ateaux
differentiable at $\ox$ with the G\^ateaux derivative representation \eqref{gI}.

To complete the proof of the theorem, it remains to show that the mapping $x(\cdot)\mapsto\ph(x(0) , x(T))$ associated with the terminal cost in \eqref{bolzap} is continuously differentiable at
$\ox(\cdot)$ in the space $\X$. To this end, observe that this mapping can be represented in the composition form $\ph\circ h$ with $h(x):=(x(0), x(T))$, which is a linear operator on $\X$.
Combining the inequalities in \eqref{normeq} and \eqref{maxnorm} tells us that $h\colon\X\to\R^n\times\R^n$ is a bounded linear mapping on $\X$ satisfying the estimates $$|h(x)|\leq\sqrt{2} \| x
\|_{\infty} \leq 2 \sqrt{2} \frac{1+T}{T} \| x \|_{\ss 1,1} \leq 2 \sqrt{2} \frac{(1 + T)^2}{T} \| x \|_{ac}\;\mbox{ for all }\;x\in\X.$$ This ensures the continuous differentiability of $h$ on $\X$ with the derivative
$\nabla h(\ox)(u) =(u(0),u(T))$ whenever $u(\cdot)\in\X$. Applying finally the classical chain rule verifies that the composition $\ph\circ h$ is differentiable on $\X$ with the derivative
$$\nabla\big[\ph\circ h\big](\ox)(u)=\big\la\nabla \ph\big(\ox(0) , \ox(T)\big),(u(0),u(T)\big)\big\ra,\quad u(\cdot)\in\X,$$ which justifies together with \eqref{gI} the claimed formula
\eqref{gJ} and thus ends the proof. \end{proof}

The next theorem deals with the constraint mapping $g\colon[0,T]\times\R^n\to\R^n$ from \eqref{bolzap} and verifies that the assumptions in (H3) imposed on $g$ ensure the fulfillment of the
smoothness assumption on the mapping $f\colon\X\to\X\times\R^n\times\R^n$ defined in \eqref{dataoc}, which is required by the necessary optimality conditions of Theorem~\ref{dualnopc}.

Prior to the formulation and proof of this result, we recall the notion of continuous embedding. Let $\X$ and $\Z$ be two normed spaces. Then $\X$ is said to be {\em continuously embedded} into
$\Z$, with the notation $ \X \hookrightarrow \Z $, if $\X \subset \Z$ and the identity mapping $i : \X \to \Z$ is continuous. For example, $L^{\infty}([0,T];\R^n) \hookrightarrow
L^{1}([0,T];\R^n)$, and by using \eqref{maxnorm} we have that $AC[0,T] \hookrightarrow L^{\infty} [0 , T]$. It is straightforward to check by definition that if $F : \Tilde\X \to\Tilde\Z$ is a
(Fr\'echet) continuously differentiable mapping and if $\X \hookrightarrow\Tilde\X $ and $\Tilde Z \hookrightarrow \Z$, then $F: \X \to \Y$ is continuously differentiable as well.

\begin{Theorem}[\textbf{Fr\'echet differentiability of the constraint mapping}]\label{Fisdiff} Let $g\colon[0, T] \times \R^n \to \R^n$ satisfy the assumptions imposed in {\rm(H4)} around a given
curve $\ox(\cdot)\in\X$. Then the mapping $f: \X \to L^1([0,T];\R^n)\times\R^n \times \R^n$ defined in \eqref{dataoc} is continuously Fr\'echet differentiable around $\ox=\ox(\cdot)$, and its
Fr\'echet derivative operator at $\ox$ is calculated by the following formula, which is valid for a.e.\ $t\in[0,T]$: \begin{equation}\label{Fg} \nabla f(\ox)u(t)=\big(\dot{u}(t) + \nabla_x g(t
,\ox(t))u(t),\;(u(0) , u(T)\big)\;\mbox{ whenever }\;u(\cdot)\in\X. \end{equation}\end{Theorem} \begin{proof} Given $g$ satisfying (H3) around the fixed curve $\ox(\cdot)\in\X$, define the
mapping $G: \X \to L^1([0,T];\R^n)$ by \begin{equation}\label{G} G(x)(t) := g\big(t , x(t)\big)\;\mbox{ for all }\;x=x(\cdot)\in\X\;\mbox{ and a.e. }\;t\in[0,T].\end{equation} We are going to
verify that $G$ is continuously differentiable around $\ox(\cdot)$ in the Fr\'echet sense with its Fr\'echet derivative at $\ox(\cdot)$ calculated by \begin{equation}\label{FG} \nabla
G(\ox)u(t)=\nabla_x g\big(t ,\ox(t)\big) u(t) \;\mbox{ a.e. }\;t\in[0,T] \end{equation} for all all $u(\cdot)\in\X$. Consider the mapping $G^\infty\colon L^{\infty}([0 ,T];\R^n)\to
L^{\infty}([0,T];\R^n)$ given by \eqref{G} but acting between different spaces in comparison with $G$. Also define the corresponding derivative mapping $D: L^{\infty}([0 ,T];\R^n) \to
L^{\infty}([0,T];\R^n)$ by \begin{equation*} D(x)(t):=\nabla_x g\big(t,x(t)\big)\;\mbox{ whenever }\;x(\cdot)\in L^\infty([0,T];\R^n).\end{equation*} We deduce from (H4) that both $G^\infty$ and
$D$ are well-defined for all $x(\cdot)$ near $\ox(\cdot)$.

All of this ensures that the assumptions in \cite[Theorem~7]{gkt} are satisfied for the case where $p=q=\infty$ therein, and thus we get by the latter result that the above mapping $G^\infty$ is
continuously Fr\'echet differentiable around $\ox(\cdot)$. The aforementioned embeddings $L^{\infty}([0,T];\R^n)\hookrightarrow L^{1}([0,T];\R^n)$ and $\X \hookrightarrow L^{\infty}([0 ,
T];\R^n)$ combined with \cite[Theorem~7]{gkt} tell us therefore that the mapping $G$ from \eqref{G} is continuously Fr\'echet differentiable around $\ox(\cdot)$ with its Fr\'echet derivative at
$\ox(\cdot)$ calculated by formula \eqref{FG}.

Considering further the constraint mapping $f: \X \to L^1([0,T];\R^n) \times\R^n \times \R^n$ defined in \eqref{dataoc} via $\dot x(\cdot)$, $g(\cdot,x(\cdot))$, and $(x(0),x(T))$. Observe that
$x(\cdot)\mapsto\dot{x}(\cdot)$ is a bounded linear mapping from $\X$ to $L^{1}([0,T];\R^n)$ due to the obvious inequality $\| \dot{x}\|_{ L^1}\le\|x\|_{ac}$. Thus it is continuously Fr\'echet
differentiable together with the mapping $x(\cdot)\mapsto(x(0),x(T))$ as shown in the proof of Theorem~\ref{fisdef}. Combining all of this with the above result for the mapping $G$ from \eqref{G}
tells us that the constraint mapping $f$ is Fr\'echet differentiable around $\ox(\cdot)$ with its derivative at $\ox(\cdot)$ calculated by \eqref{Fg}. This completes the proof. \end{proof}

The last result of this section concerns the calculation of the normal cone $N_\Th$ in the necessary optimality conditions of Theorem~\ref{dualnopc} for the convex set $\Th$ defined in
\eqref{dataoc} via the initial data of the extended Bolza problem \eqref{bolzap}. In fact, the structure of the set $\Th$ in \eqref{dataoc} suggests that it suffices to calculate the normal cone
to the set \begin{equation}\label{hatc}\Th_1:=\big\{y\in L^1([0,T];\R^n)\;\big|\;y(t)\in\O_1\;\mbox{ a.e. }\;t\in [0,T]\big\},\end{equation} where $\O_1$ is a closed and convex subset of $\R^n$.
Indeed, from the calculation of $N_{\Th_1}$ we immediately come to the required formula for the normal cone to the set $\Th$ in question by the elementary calculus rule for normals to set
products in convex analysis.

\begin{Theorem}[\textbf{normal cone calculation for the constraint set}]\label{nconecal}  Let $\oy=\oy(\cdot)\in L^1([0,1];\R^n)$ be such that $\oy(t)\in\O_1$, where the set $\O_1\subset\R^n$ is
convex and locally closed around $\oy(t)$ for a.e. $t\in[0,T]$. Then we have the calculation formula \begin{equation}\label{nconecal1} N_{\Th_1} (\oy) =\big\{ p \in L^{\infty}([0,T];\R^n) \;
\big|\; p(t)\in N_{\O_{1}}(\oy(t))\;\mbox{ a.e. }\;t\in [0,T]\big\} \end{equation} for the normal cone to the set $\Th_1$ defined in \eqref{hatc}.\end{Theorem} \begin{proof} Pick any $p \in
L^{\infty}([0,T];\R^n)$ such that $p(t) \in N_{\Th_1} (\oy (t))$ for a.e. $t \in [0,T]$. Taking an arbitrary function $y(\cdot) \in\Th_1$ and using the normal cone definition \eqref{nc} we get
that $\la p(t),y(t)-\oy(t)\ra\le 0 $ for a.e. $t \in [0,T]$. This leads us, by using the canonical pairing between $L^1([0,T];\R^n)$ and the dual space  $L^\infty([0,T];\R^n)$, to $$\la p \;, \;
y-\oy \ra = \int_{0}^{T} \la p(t) , y(t) - \oy (t) \ra\,dt \leq 0,$$ which yields $p \in N_{\Th_1} (\oy)$. To verify the opposite inclusion, fix $ p \in N_{\Th_1} (\oy)$ and $t \in (0, T)$.
Taking a countable dense subset $Q:= \{c_i\}_{i=1}^{\infty}$ of $\O_1$, pick any $c_i\in Q $ and choose $r=r(t) > 0$ to be so small that $(t-r , t+r) \subset (0,T)$. Define now $y^i : [0,T] \to
\R^n$ by $$y^i (s) :=\left\{\begin{matrix} c_i & s \in (t-r , t+r), \\
 \oy(s)& s \notin (t-r , t+r)
\end{matrix}\right.$$ and easily observe that $y^i(\cdot)\in\Th_1$. Thus we get $$ 0 \geq \la p , y^i - \oy \ra  = \int_{0}^{T} \la p(s) , y^i (s) -\oy(s) \ra\,ds = \int_{t-r}^{t+r} \la p(s) ,
c_i - \oy(s) \ra\,ds. $$ Then basic real analysis tells us that $$ \frac{1}{2r} \int_{t-r}^{t+r} \la p(s) , c_i - \oy(s) \ra\,ds \to \la p(t) , c_i - \oy(t) \ra \quad\mbox{as} \quad r \dn 0 $$
for all $t\in [0,T] \setminus S_i$ with $S_i$ being of zero measure. This implies that for all $i$ we have $$ \la p(t),c_i - \oy(t) \ra \le 0\;\mbox{ whenever }\;t \notin S_i. $$ Consider further
the set $S:=\bigcup_{i=0}^{ \infty} S_i$, which is also of zero measure, where $S_0 : = \{ t \in [0,T]\;| \; \oy(t) \notin\O_1\}$. Hence we have $$ \la p(t),c_i - \oy(t) \ra \le 0\;\mbox{ for all
}\; i\in\{1,2,\ldots\}\;\mbox{ and }\;t \notin S.$$ Since the set $Q$ is dense in $\O_1$, it follows from the above that $p(t)\in N_{\O_1}(\oy(t))$ for a.e. $t\in[0,T]$, which completes the proof
of the theorem.\end{proof}


\section{Derivation of Necessary Conditions}\label{sec:proof}

Having in hand the above results supporting the reduction of the extended Bolza problem \eqref{bolzap} to the nondynamic constrained optimization \eqref{op}, we derive in this section the
necessary optimality conditions for \eqref{bolzap} formulated in Theorem~\ref{bolza} from those obtained in Theorem~\ref{dualnopc} for problem \eqref{op} in general normed spaces.\vspace*{0.05in}

In our derivation we need the following extended version of the {\em fundamental lemma of the calculus of variations} that we were not able to find in the literature.

\begin{Lemma}[\textbf{extended fundamental lemma of the calculus of
variations}]\label{flemma} Let $a,b \in \R^n$, and let
$l(\cdot),q(\cdot)\in L^1([0,T];\R^n)$. Assume that
\begin{equation}\label{flemma1} \int_{0}^{T} \langle l(t), h(t)
\rangle\,dt +  \int_{0}^{T}  \langle q(t) , \dot{h} (t) \rangle\,dt
+ \langle h(0) , a \rangle + \langle h(T) , b \rangle  =
0\end{equation} for all $h(\cdot)\in C^{\infty}([0,T];\R^n).$ Then
there exists a unique function $\oq(\cdot)\in AC([0,T];\R^n)$ such
that $q(t) =\oq(t)$ for a.e. $t\in[0,T]$ with $$\oq (0) =
a,\;\oq(T)= - b,\;\mbox{ and }\;l(t)=\dot{\oq}(t)\;\mbox{ a.e.}
\;t\in [0,T] .$$ Moreover, the function $\oq(t)$ can be determined
by \begin{equation}\label{oq} \oq(t) = \frac{d}{dt} \int_{0}^{t}
q(s)\,ds\;\mbox{ whenever }\;t\in [0,T]. \end{equation} \end{Lemma}
\begin{proof} We may assume for simplicity that all the functions
under consideration are real-valued (not vector-valued), because in
the vector setting the same arguments can be applied to the
components of these functions. To begin with, let us first use
\eqref{flemma1} for smooth functions $h(\cdot)$ with $h(0)=h(T)=0$.
Since $l(\cdot)\in L^1([0,T];\R)$, the integral $$
L(t):=\int_0^tl(s)ds$$ is absolutely continuous on $[0,T]$.
Integrating the first term in \eqref{flemma1} by parts gives us
$$\int_0^T\big\la (q-L)(t), \dot{h}(t)\big\ra\,dt=0\;\mbox{ for all
such }\;h(\cdot).$$ This implies that $q(\cdot)-L(\cdot)$ is a
constant function a.e.\ on $[0,T]$, which allows us to find a real
number $\gg\in \R$ such that $q(t) = L(t)+\gg$ for a.e. $t \in
[0,T]$. Defining $\oq(t):=L(t) +\gg$ for all $t\in[0,T]$, we
immediately get $\oq(\cdot)\in AC([0,T];\R)$ with $\dot{\oq}
(t)=l(t)$ and $\oq(t) = q(t)$ a.e.\ on $[0, T]$. The latter allows
us to conclude that \eqref{flemma1} holds for all $h(\cdot)\in{\cal
C}^{\infty}([0 ,T];\R)$ if we replace $q(\cdot)$ by $\oq(\cdot)$
therein. This leads us to the equality $$\big\langle
h(T),\big(b+\oq(T)\big)\big\rangle +\big\langle
h(0),\big(a-\oq(0)\big)\big\rangle=0\;\mbox{ for all
}\;h(\cdot)\in{\cal C}^{\infty}([0 ,T];\R),$$ which yields
$\oq(0)=a$ and $\oq(T)=-b$. Observe that the integral mapping $ t
\mapsto \int_{0}^{t}\bar q(s)\,ds $ is differentiable by the
fundamental theorem of calculus. Thus for all $t \in [0,T]$ we have
$$\oq(t) = \frac{d}{dt} \int_{0}^{t} \oq(s)\,ds = \frac{d}{dt}
\int_{0}^{t}q(s)\,ds, $$ which verifies \eqref{oq} and completes the
proof of the lemma. \end{proof}

It is worth mentioning that if the functions $l(\cdot)$ and $q(\cdot)$ satisfy the assumptions of Lemma~\ref{flemma}, then any pointwise perturbations of them on a measure zero subset of $[0,T]$
also satisfy these assumptions. This tells us that there are many functions $q(\cdot)$ satisfying \eqref{flemma1}, which are not absolutely continuous on $[0,T]$, and thus they differ from their
(unique) absolute continuous representatives.\vspace*{0.05in}

Now we are in a position to prove Theorem~\ref{bolza} by combining the above results on the reduction of the extended Bolza problem \eqref{bolza} to problem \eqref{op} of nondynamic optimization
in normed spaces with the usage of some other tools of variational analysis.\\[1ex] {\bf Proof of Theorem~\ref{bolza}.} Fix any normed space $\X$ of functions $x\colon[0,T]\to\R^n$ satisfying the
inclusions in \eqref{space}, and let $\Y := L^1([0 ,T];\R^n)\times\R^n\times\R^n$. As discussed in Section~\ref{conic}, the extended Bolza problem \eqref{bolzap} can be written in the nondynamic
form \eqref{op} of constrained optimization with $J\colon\X\to\R$, $f\colon\X\to\Y$, and $\Th\subset\Y$ defined in \eqref{dataoc}. Let us confirm that all the assumptions of
Theorem~\ref{dualnopc} hold for the initial data \eqref{dataoc} generated by the extended Bolza problem under the assumptions in (H1)--(H5). Indeed, the convexity of $\Th=\Th_1\times\O_2$, with
$\Th_1$ defined in \eqref{hatc}, immediately follows from the convexity of $\O_1$ and $\O_2$ imposed in (H3). The local closedness of $\Th\subset\Y$ around $f(\ox(\cdot))$ follows from the
closedness of $\O_1$, $\O_2$ in (H3) and the structure of $\Th$ in \eqref{dataoc} due the classical result of real analysis telling us that the (norm) convergence of the sequence in
$L^1([0,T];\R^n)$ yields the a.e.\ convergence of a subsequence on $[0,T]$. The G\^ateaux differentiability of the cost functional $J$ at $\ox$ and its local Lipschitz continuity around this
point under the assumptions in (H1) and (H2) follow from Theorem~\ref{fisdef} applied to the $\X$-strong minimizer $\ox=\ox(\cdot)\in\X$. Furthermore, the continuous Fr\'echet differentiability
of the constraint mapping $f$ from \eqref{dataoc} is proved in Theorem~\ref{Fisdiff} under the assumptions imposed in (H4) on the $\X$-strong local minimizer $\ox(\cdot)$ of \eqref{bolza}.

To finish checking the assumptions of Theorem~\ref{dualnopc} with the data taken from \eqref{dataoc}, let us show that the imposed qualification condition in (H5) is equivalent to the metric
subregularity constraint qualification \eqref{mscq} for the reduced problem \eqref{op} with the same modulus $\kappa>0$. To proceed, pick $x(\cdot)\in \X$ and get the equalities
\begin{eqnarray*}&&\disp \dist\big(\dot{x}(\cdot) +  g(\cdot, x(\cdot));\Th_1\big)\\ &&=\inf \Big\{ \int_{0}^{T}\Big\|\;\dot{x}(t) + g(t , x(t)) - y(t)\|\,dt\Big| \; y(\cdot)\in
L^1([0,T];\R^n),\;y(t) \in\O_1\;\mbox{ a.e.}\; t \in [0, T]\Big\}\\&&=\inf \Big\{ \int_{0}^{T}\Big\|\;\dot{x}(t) + g(t , x(t)) - y(t)\|\,dt\Big| \; y(\cdot)\in L^1([0,T];\R^n),\;y(t)
\in\O_1\;\mbox{ for all }\; t \in [0, T]\Big\}\\ &&= \inf_{y(\cdot) \in L^1([0,T];\R^n)}\int_{0}^{T} \| \dot{x}(t) +  g(t , x(t)) - y(t) + \delta_{\O_1} (y(t))  \|\,dt \\ &&=\int_{0}^{T}
\inf_{y\in \R^n} \| \dot{x}(t) + g(t , x(t)) - y + \delta_{\O_1} (y)\|\,dt \\ &&=\int_{0}^{T}\inf_{y\in C_1}  \| \dot{x}(t) +  g(t , x(t)) - y \|\, dt = \int_{0}^{T}  \dist \big( \dot{x}(t) + g(t
, x(t));\O_1\big)\, dt, \disp \end{eqnarray*} where we interchange the integral and infimum signs by using \cite[Theorem~14.60]{rw}. This tells us that (H5) can be equivalently written as
\begin{eqnarray*} \disp \dist(x(\cdot);{\cal S}) & \leq &\kappa\int_{0}^{T} \dist\big(\dot{x}(t) + g(t, x(t));\O_1\big)\, dt + \kappa\,\dist\big((x(0) , x(T));\O_2) \\& \leq &\kappa\,\dist \big(
\dot{x}(\cdot) + g(\cdot , x(\cdot));\Th_1 \big) + \kappa\,\dist\big( (x(0) , x(T));\O_2\big)=\kappa \dist\big(f(x(\cdot));\Theta).\disp \end{eqnarray*}

Therefore, all the assumptions of Theorem~\ref{dualnopc} are satisfied for the nondynamic version \eqref{op} of the extended Bolza problem \eqref{bolzap}, and now we can apply to the local
minimizer $\ox=\ox(\cdot)$ of \eqref{op} the necessary optimality conditions \eqref{necessaryop} in the case where the data of \eqref{op} are given by \eqref{dataoc}. Note that the local
minimizer $\ox$ of \eqref{op} corresponds to the $\X$-strong local minimizer $\ox(\cdot)$ in the sense of \eqref{st-min}.

According to \eqref{necessaryop} in our setting, there exists a multiplier $\lm(\cdot)=(\mu(\cdot),s_1,s_2)\in \Y^* = L^{\infty}([0,T];\R^n)\times\R^n \times \R^n$ such that for all
$u(\cdot)\in\X$ we have \begin{equation}\label{bolzaeq3} \left\{\begin{matrix} \nabla J\big(\ox(\cdot)\big)(u(\cdot)\big)+\big\la\nabla f\big(\ox(\cdot)\big)\big(u(\cdot)\big)
,\big(\mu(\cdot),s_1,s_2)\big)\big\ra   = 0, \\\\ \mu(\cdot) \in N_{\Th_1}\big(\dot{\ox}(\cdot) + g(\cdot,\ox(\cdot))\big),\;\mbox{ and }\;(s_1 ,s_2) \in  N_{\O_2}\big(\ox(0) , \ox(T)\big), \\
\end{matrix}\right. \end{equation} where the second line follows from the application to the product set $\Th=\Th_1\times\O_2$ the normal cone product formula \begin{equation*}
N_\Th\big(y(\cdot),z\big)=N_{\Th_1}\big(y(\cdot)\big)\times N_{\O_2}(z)\;\mbox{ for all }\;y(\cdot)\in\Th_1\;\mbox{ and }\;z\in\O_2.\end{equation*} Using the calculation of the normal cone to
$\Th_1$ given in Theorem~\ref{nconecal} tells us that the condition $\mu(\cdot)\in N_{\Th_1}(\dot{\ox}(\cdot) + g(\cdot ,\ox(\cdot)))$ is equivalent to \begin{equation}\label{mu}\big\la \mu (t),
\dot{\ox}(t)+ g\big(t,\ox(t)\big)\big\ra\geq\la\mu (t),w\ra\;\mbox{ for all }\;w\in\O_1\;\mbox{ and a.e. }\;t\in [0,T]. \end{equation} Furthermore, using the calculations of the G\^ateaux
derivative of $J$ and the Fr\'echer derivative of $f$ obtained in Theorem~\ref{fisdef} and Theorem~\ref{Fisdiff}, respectively, allows us to rewrite the first line of \eqref{bolzaeq3} in the form
\begin{eqnarray*}&&\int_{0}^{T}\bigg(\Big\la\nabla_x\vt\big(t , \ox (t) , \dot{\ox}(t)\big) + \nabla_x g\big(t , \ox(t)\big)^{*}\mu (t),u(t)\Big\ra+\Big\la\nabla_v\Vt\big(t , \ox (t) ,
\dot{\ox}(t)\big) + \mu (t),\dot{u}(t)\Big\ra\bigg)dt \\&&\;+ \big\la u(0),s_1 +\nabla_{x_0}\ph_1\big(\ox(0),\ox(T)\big)\big\ra+\big\la u(T),s_2 +\nabla_{x_T}\ph\big(\ox(0),\ox(T)\big)\big\ra
=(0,0). \end{eqnarray*} Since the above relationships hold for all $u(\cdot)\in\X$, appealing to Lemma~\eqref{flemma} leads us to the Euler-Lagrange equation \eqref{euler-lag} together with the
endpoint conditions \begin{equation}\label{bolzaeq5} \left\{\begin{matrix} p(0) = s_1 + \nabla_{x_0}\ph\big(\ox(0),\ox(T)\big),  \\ \quad p(T) = - s_2 -\nabla_{x_T}\ph\big(\ox(0) , \ox(T)\big),
\end{matrix}\right. \end{equation} where $p : [0,T] \to \R^n $ is the unique absolutely continuous representative of the mapping $t\mapsto\nabla_v\Vt(t, \ox(t) , \dot{\ox}(t) ) + \mu (t)$.
Remembering that $(s_1, s_2) \in N_{\O_2} (\ox(0) , \ox(T))$, the transversality inclusion \eqref{transv} is implied by \eqref{bolzaeq5}. Observe furthermore that $$ \mu (t)= p(t) -
\nabla_{v}\Vt\big(t, \ox(t) , \dot{\ox}(t)\big)\;\mbox{ for a.e. }\;t \in [0,T]. $$ Replacing $\mu(\cdot)$ in \eqref{mu} by the latter expression, we arrive at the Weierstrass-Pontryagin
maximization condition \eqref{maxp} and thus complete the proof of the theorem. $\hfill\Box$

\begin{Remark}[\textbf{quantitative relationships in optimality conditions}]\label{quant} {\rm The necessary optimality conditions of Theorem~\ref{bolza}, while being based on the results of
Theorem~\ref{dualnopc} for nondynamic constrained optimization in normed spaces, do not explore the novel {\em quantitative} condition of the latter theorem that gives us an efficient estimate of
the multiplier $\lm$ in terms of the problem data of \eqref{op}. Our intention is to utilize this estimate in deriving explicit qualitative relationships of this type for the extended Bolza
problem \eqref{bolzap}. This seems to be important and implementable within our approach, but requires some technical work, which will be done in our future research.}\end{Remark}

Finally, we present an example that contains a simple class of
variational problems, where all the assumptions of
Theorem~\ref{bolza} are satisfied and the obtained necessary
optimality conditions are explicitly formulated.

\begin{Example}[\bf illustrating optimality conditions]{\rm We
consider the following problem of the calculus of variations:
\begin{eqnarray*} \disp &\mbox{minimize}& ~~ \quad \ph\big(x(0) ,
x(T)\big) + \int_{0}^{T}\Vt\big(t , x(t)\big)\,dt  \\ \nonumber
&\mbox{subject to}& ~~ \quad x(0) \in \Theta_1 , ~ x(T) \in \Th_2,\\
\nonumber &\quad & ~~ \quad x(\cdot) \in AC([0,T],\R^n), \disp
\end{eqnarray*} where $\ph$ and $\vt$ are continuously differentiable, thus they satisfy assumption (H1) and (H2) for all $ \ox \in AC([0,T],\R^n)$. Further, we assume $\Th_{1}$ and $\Theta_{2}$ are polyhedral convex sets. It is easy to see that the set of feasible solutions to this problem is also
polyhedral in $ AC([0,T],\R^n)$. Using an appropriate infinite-dimensional extension of the seminar Hoffman's lemma from \cite[Theorem~6a]{bt} tells us that assumption (H5) holds for $\O_{2}:= \Th_{1} \times \Th_{2}$. Then the
necessary optimality conditions of Theorem~\ref{bolza} ensure the
existence of a dual arc $p(\cdot) \in AC([0,T],\R^n)$
satisfying the following relationships:\\[1ex] The {\sc Euler-Lagrange equation} for a.e. $t\in[0,T]$: \begin{equation*}\label{euler-lag}\dot{p}(t) =\Vt_x \big(t,x(t)\big). \end{equation*}
The {\sc transversality inclusion}:
\begin{equation*}\label{transv}\big(p(0),-p(T)\big)\in\nabla\ph\big(
x(0), x(T)\big)+N_{\O_2}\big( x(0), x(T)\big). \end{equation*} The
Euler-Lagrange equation reads as $p(t) =\int_{0}^{t}\Vt_x \big(s ,
x(s)\big)\,ds + p(0) $ for all $t \in [0,T]$. Plugging there $t=T$,
we get the equality $$ p(0)- p(T)  = -\int_{0}^{T} \Vt_x \big(t , x(t)\big)\,dt .$$ Employing finally the transversality inclusion together with the fact that $N_{\O_2}\big(\ox(0), \ox(T) \big) = N_{\Th_1}\big(\ox(0)\big) \times N_{\Th_2}\big(\ox(T)\big)$
confirm that any minimizer $\ox(\cdot)$ of the above problem has to
satisfy the following explicit condition
$$\int_{0}^{T}\Vt_x
\big(t,\ox(t)\big)\,dt + \nabla_{x_0}\ph\big(\ox(0),\ox(T)\big)
+\nabla_{x_T}\ph\big( \ox(0),\ox(T)\big) \in -
N_{\Th_1}\big(\ox(0)\big) - N_{\Th_2}\big(\ox(T)\big)$$ which allows us to eliminate
nonoptimal solutions and eventually calculate local minimizers for
specified initial data of the problem under consideration.}
\end{Example}


\section{Conclusions}\label{concl} This paper develops a new approach to study problems of dynamic optimization in generally incomplete spaces by reducing them to nondynamic problems of
constrained optimization in normed spaces with subsequent applications of the refined necessary optimality conditions recently obtained \cite{mm}. We implement here this approach to deriving
first-order necessary optimality conditions in the extended Bolza problem of the calculus of variations with pointwise velocity constraints depending on current state positions. In contrast to
the previously used methods of modern variational analysis that require the completeness of spaces of feasible solutions, we are now able to deal with both complete space frameworks as well as
with incomplete spaces of type ${\cal C}^k$ for $k=1,2,\ldots$. Investigating strong local minimizers in the spaces under consideration, we confined ourselves for simplicity to variational
problems with smooth data although suitable constructions of generalized differentiation in normed spaces are used even in such settings. Proceeding in this way, we derive necessary optimality
conditions for $\X$-strong local minimizers of the extended Bolza problem in generally incomplete spaces that contain the appropriate Euler-Lagrange, Weierstrass-Pontryagin, and transversality
relations.

In our future research, we plan to develop this approach with covering nonsmooth and nonconvex extended Bolza problems as well as optimal control problems for constrained differential inclusions.
One of the important novel features of this approach is the possibility to obtain some {\em quantitative estimates} for adjoint functions as discussed in Remark~\ref{quant}. We also plan to
implement this approach to deriving {\em second-order} optimality conditions for variational problems by extending to infinite dimensions and further developing the recent results of second-order
variational analysis achieved in \cite{mms,mms1}.


\small \begin{thebibliography}{11}

\bibitem{aubin} Aubin J-P, Cellina A (1984) Differential Inclusions.
Springer, Berlin

\bibitem{bliss} Bliss GA (1946) Lectures on the calculus of variations. The University of Chicago Press, Chicago, IL

\bibitem{bolza} Bolza O (1904) Lectures on the calculus of variations. The University of Chicago Press, Chicago, IL

\bibitem{bt} Burke JV, Tseng P (1996) Unified analysis of Hoffman's bound via Fenchel duality. SIAM J Optim 6:265--282

\bibitem{c} Clarke F (2013) Functional analysis, calculus of variations and optimal control. Springer, London

\bibitem{g} Gfrerer H (2011) First-order and second-order characterizations of metric subregularity and calmness of constraint set mappings. SIAM J Optim 21:1439--1474

\bibitem{gm} Gfrerer H, Mordukhovich BS (2015) Complete characterizations of tilt stability in nonlinear programming under weakest qualification conditions. SIAM J Optim 25:2081--2119

\bibitem{go} Gfrerer H, Outrata JV (2016) On computation of generalized derivatives of the normal-cone mapping and their applications. Math Oper Res 41:1535--1556

\bibitem{gkt} Goldberg H, Kampowsky W, Troltzsch F (1992) On Nemytskij operators in LP spaces of abstract functions. Math Nachr 155:127--140

\bibitem{i} Ioffe AD (2017) Variational analysis of regular mappings: theory and applications. Springer, Cham, Switzerland

\bibitem{mm} Mohammadi A, Mordukhovich BS (2021) Variational analysis in normed spaces with applications in constraint optimization. SIAM J Optim 31:569--603

\bibitem{mms1} Mohammadi A, Mordukhovich BS, Sarabi ME (2021) Parabolic regularity in geometric variatiional analysis. Trans Amer Math Soc 374:1711--176

\bibitem{mms} Mohammadi A, Mordukhovich BS, Sarabi ME (2021)
Variational analysis of composite models with applications to
continuous optimization. Math Oper Res (to appear). DOI:
10.1287/moor.2020.1074

\bibitem{m93} Mordukhovich BS (1993) Complete characterizations of openness, metric regularity, and Lipschitzian properties of multifunctions. Trans Amer Math Soc 340:1--35.

\bibitem{m06} Mordukhovich BS (2006) Variational analysis and generalized diffrerentiation, I: basic theory, II: applications. Springer, Berlin

\bibitem{pont} Pontryagin LS, Boltyanskii VG, Gamkrelidze RV, Mishchenko RF (1962). The mathematical theory of optimal processes. Wiley, New York

\bibitem{rw} Rockafellar RT, Wets RJ-B (1998) Variational analysis. Springer, Berlin

\bibitem{s} Smirnov GV (2002) Introduction to the theory of differential inclusions. American Mathematical Society, Providence, RI

\bibitem{tonelli} Tonelli L (1921, 1923) Fondamenti di calcolo delle variazoni, I, II. Nicola Zanichelli, Bologna

\bibitem{v} Vinter RB (2000) Optimal control. Birkh\"auser, Boston, MA

\end{thebibliography}
 \end{document}